\documentclass[12pt]{amsart}

\usepackage[initials]{amsrefs}
\usepackage{amssymb,amsmath,amsfonts,latexsym,amsthm,paralist,graphicx}

\usepackage[a4paper,nomarginpar]{geometry}

\usepackage[english]{babel}

\newtheorem{theorem}{Theorem}[section]
\newtheorem{corollary}[theorem]{Corollary}
\newtheorem{proposition}[theorem]{Proposition}
\newtheorem{lemma}[theorem]{Lemma}

\theoremstyle{definition}

\newtheorem{example}[theorem]{Example}

\newtheorem{remark}[theorem]{Remark}
\newtheorem{definition}[theorem]{Definition}

\newcommand{\N}{\mathbb{N}} 
\newcommand{\Q}{\mathbb{Q}} 
\newcommand{\R}{\mathbb{R}} 
\newcommand{\K}{\mathbb{K}}

\newcommand{\card}{\operatorname{card}}
\newcommand{\supp}{\operatorname{supp}}

\begin{document}

\title[Anti M-Weierstrass function sequences]{Anti M-Weierstrass function sequences}

\date{}

\author[Calder\'on]{Mar\'{\i}a del Carmen Calder\'on-Moreno}
\author[Gerlach]{Pablo Jos\'e Gerlach-Mena}
\author[Prado]{Jos\'e Antonio Prado-Bassas}
\address{Dpto. An\'{a}lisis Matem\'{a}tico,
\newline\indent Facultad de Matem\'{a}ticas,
\newline\indent Universidad de Sevilla, \newline\indent Avda. Reina Mercedes s/n,\newline\indent 41080 Sevilla, Spain.}
\email{mccm@us.es, gerlach@us.es, bassas@us.es}

%
%

\thanks{The authors
have been partially supported by the Plan Andaluz de Investigaci\'on de la Junta de Andaluc\'{\i}a FQM-127 and by MICINN Grant PGC2018-098474-B-C21.}

\begin{abstract}
Large algebraic structures are found inside the space of sequences of continuous functions on a compact interval having the property that, the series defined by each sequence converges absolutely and uniformly on the interval but the series of the upper bounds diverges. So showing that there exist many examples satisfying the conclusion but not the hypothesis of the Weierstrass M-test.
\end{abstract}

\maketitle

\section{Introduction}

In any field of Mathematics, whenever the uniform convergence of a series of functions $\sum_{n=1}^\infty f_n(x)$ must be studied (where $f_n:X\to \Bbb K$, $X\ne\varnothing$ and $\Bbb K=$ the real line $\Bbb R$ or the complex plane $\Bbb C$), the tool we first think about is the well-known Weierstrass M-test. If there exists a sequence $(M_n)_n$ of positive numbers such that its series is convergent and $|f_n(x)|\le M_n$ for every $x\in X$ and every $n\in\Bbb N$ ($:=$the set of all positive integers), then the series $\sum_{n=1}^\infty f_n(x)$ is absolutely and uniformly convergent on $X$ (see, for instance, \cite[\S 9.6]{Apostol}).

The converse is false in general. Concretely, the hypothesis about the majorant sequence can be dropped without avoiding the uniform convergence of the series. Indeed, for every $n\in \Bbb N$ and every $x\in[0,1]$, define \begin{equation}\label{Ex:basic}f_n(x)=\begin{cases}
  \frac{1}{n}\sin^2\left(2^{n+1}\pi x\right)& \text{if }x\in\left(2^{-(n+1)},2^{-n}\right)\\
  0&\text{elsewhere.}
\end{cases}\end{equation}
The sequence $f_n(x)$ cannot be majorated by any summable sequence of positive numbers, since $\sup_{x\in[0,1]} |f_n(x)|=\frac{1}{n}$; however, the series $\sum_{n=1}^\infty f_n(x)$ is absolutely and uniformly convergent (see \cite[Chapter 1, Example 10]{bourchtein} for more details).

Such a sequence of functions will be called {\em Anti M-Weierstrass} (see Definition \ref{Def:AntiM} below).

We denote by ${C}([a,b])^{\Bbb N}$ the set of all sequences of continuous functions on a nondegenerate compact interval $[a,b]\subset\Bbb R$. Recall that if we endow $C([a,b])^{\Bbb N}$ with the product topology inherited from $\left(C([a,b]),\|\cdot\|_\infty\right)$, where $\|f\|_\infty =\sup_{x\in[a,b]}|f(x)|$ is the classical supremum norm, we get a complete metrizable and separable topological vector space (see \cite[Chapter 7]{Willard}).


\begin{definition}\label{Def:AntiM}
  Let $(f_n)_n\in C([a,b])^{\Bbb N}$ be a sequence of continuous functions on $[a,b]$. We say that $(f_n)_n$ is an {\em Anti M-Weierstrass sequence} if $\sum_{n=1}^\infty f_n(x)$ is absolutely and uniformly convergent on $[a,b]$, but $\sum_{n=1}^\infty \|f_n\|_\infty$ diverges.
\end{definition}
We will denote by $\mathcal{AMW}([a,b])$ (or just $\mathcal{AMW}$, whenever there is no possibility of confusion) {\it the family of Anti M-Weierstrass sequences on $[a,b]$}.

\medskip

The aim of this paper is constructing large linear structures of Anti M-Weiers\-trass sequences, and this objective lies in the branch of Lineability, a trend of research that has attracted the attention of many mathematicians during the past few decades. We introduce in the next paragraph the main notions about Lineability we are going to deal with. We refer the interested reader to \cite{librolineabilidad, survey} for a general background.

Given a vector space $X$ and a (finite or infinite) cardinal number $\kappa$, we say that a subset $A$ of $X$ is $\kappa$-lineable whenever there is a vector space $M$ of dimension $\kappa$ such that $M \setminus \{0\} \subset A$; when $\kappa=\dim(X)$, we speak about maximal-lineability. Moreover, if $X$ is a topological vector space, then $A$ is said to be spaceable (dense-lineable, $\kappa$-dense-lineable, maximal-dense-lineable, resp.) in $X$, if there is a closed infinite dimensional (a dense, a dense $\kappa$-dimensional, a dense $\dim(X)$-dimensional, resp.) vector space $M$ with $M\setminus\{0\}\subset A$. When, in addition, $X$ is contained in some (linear) algebra, then $A$ is called $\kappa$-algebrable if there is an algebra $M$ so that $M \setminus \{0\} \subset A$ and the cardinality of any minimal system of generators of $M$ is $\kappa$; if the algebra can be taken to be free, we say that $A$ is strongly $\kappa$-algebrable; and if the (free) algebra is dense in $X$, we speak about (strongly) $\kappa$-dense-algebrability. If the algebraic structure of $X$ is commutative, the strong algebrability is equivalent to the existence of a generating system $B$ with $\card(B)=\kappa$ of the algebra $M$ such that for any $N \in \N$, any nonzero polynomial $P$ in $N$ variables without constant term and any distinct $b_1,\ldots,b_N \in B$, we have that $P(b_1, \ldots ,b_N) \in M\setminus\{0\}$.

Concerning lineability of series, in 2005 Bayart \cite{Bayart2005} showed that the set of continuous functions on the unit circle $\Bbb T$ whose Fourier series diverges on a set $E\subset\Bbb T$ of measure zero is dense-lineable. A year later, Aron, P\'erez-Garc\'{\i}a and Seoane-Sep\'ulveda \cite{ArPGSeo} stated the dense-algebrability. When $E$ is countable, results on lineability of divergent Fourier series with additional properties are obtained in \cite{Bernal2012} or \cite{Muller2010}.

Again Bayart \cite{Bayart2005o,Bayart2005} showed that the set of Dirichlet series $f(s)=\sum_{n=1}^\infty a_nn^{-s}$ that are bounded in the right half-plane and diverge everywhere in the imaginary axis is lineable and spaceable. Together with Quarta \cite{BayQua}, they were able to establish also the algebrability.

In the setting of Banach spaces, Aizpuru, P\'erez-Eslava and Seoane-Sep\'ulveda in 2006 \cite{AizPESeo}, asserted that the set of unconditionally convergent but not absolutely convergent series of any infinite dimensional Banach space is $\mathfrak c$-lineable (where  $\mathfrak c$ denotes the cardinality of continuum). Moreover, the set of sequences such that its partial sums are bounded but the series diverges is also $\mathfrak c$-lineable. If we focus on the complex plane, Bartoszewicz, G\l{a}b and Poreda proved in \cite{BGP} that the set of nonabsolutely convergent complex series and the set of divergent complex series with bounded partial sums are $\mathfrak{c}$-algebrable.

Following with sequences of real numbers, in 2013 Bartoszewicz and G\l{a}b  \cite{BarGlab13} showed that $c_0\setminus\bigcup_{p\ge1}\ell_p$ is densely strongly $\mathfrak{c}$-algebrable, where $c_0$ denotes the set of all real sequences converging to zero, and $\ell_p$ is the set of all $p$-summable real sequences.

Later, in 2017, Ara\'ujo {\it et al}.~\cite{ABMPS} studied the linear structure of the family of sequences of real numbers whose corresponding series fail both the root and ratio test. Specifically, they show that the set of sequences in $\ell_1$ generating series for which the ratio or the root tests fail and the set of sequences in $\omega$ (the vector space of all real sequences endowed with the product topology) generating divergent series for which the ratio and the root tests fail are $\mathfrak{c}$-dense-lineable in $\ell_1$ or $\omega$.

In this paper we are concerned with lineability --in its several degrees-- of the family $\mathcal{AMW}$ described in Definition \ref{Def:AntiM}. In Section 2 we shall define a special family of sequences of functions that will be crucial in the forthcoming constructions and provide several auxiliar results about it. Section 3 is devoted to construct several vector spaces inside $\mathcal{AMW}$ and to get its maximal dense lineability and spaceability. Finally, in Section 4, we conclude the paper by establishing, as a consequence of a more general result, the strong $\mathfrak{c}$-algebrability of our family $\mathcal{AMW}$.

\section{Concepts and Examples}

Inspired by the Example (\ref{Ex:basic}) of Anti M-Weierstrass sequence provided in the introduction, we can establish a very useful auxiliary result. By $\supp(f)$ we denote the support of a function $f :[a,b] \to \K$, that is, the set of points $x \in [a,b]$ where $f$ does not vanish.

Let $\mathcal{F}$ be the family of all sequences of functions $(u_n)_n \in C([a,b])^{\N}$ such that:
\begin{itemize}
\item The supports are pairwise disjoint, that is,
\begin{equation}
\supp(u_n) \cap \supp(u_m) = \varnothing, \qquad  n \neq m.
\end{equation}
\item The sequence $(u_n)_n$ is both uniformly bounded and uniformly far from zero, that is,
\begin{equation}
0< \inf_n ||u_n||_\infty \leq \sup_n ||u_n||_\infty <+ \infty.
\end{equation}
\end{itemize}


\begin{lemma}\label{FA}
Let $(u_n)_n \in \mathcal{F}$ and let $(a_n)_n \subset \R$. Then we have:
\begin{itemize}
\item[{\rm (a)}] The series $\displaystyle{ \sum_{n=1}^\infty a_n u_n(x)}$ converges absolutely on $[a,b]$.

\item[{\rm (b)}] The series $\displaystyle{ \sum_{n=1}^\infty a_n u_n(x)}$ converges uniformly on $[a,b]$ if and only if $(a_n)_n \in c_0$.

\item[{\rm (c)}] The series $\displaystyle{ \sum_{n=1}^\infty \| a_n u_n \|_\infty} < +\infty$ if and only if $(a_n)_n \in \ell_1$.

\end{itemize}
\end{lemma}

\begin{proof}

(a)  The absolute convergence of the series is immediate, since the disjointness of the supports of the $u_n$'s implies that, for a fixed $x_0 \in [a,b]$, either $u_n(x_0) =0$ for all $n \in \N$, or there exists only one $n_0 \in \N$ such that $x_0 \in \supp(u_{n_0})$, and
    $$\sum_{n=1}^\infty |a_n u_n(x_0)| = |a_{n_0} u_{n_0}(x_0)| \quad (< + \infty).$$

(b)  Firstly, assume that $(a_n)_n \in c_0$. Because $(u_n)_n \in \mathcal{F}$, then $M := \sup_n ||u_n||_\infty \in (0, +\infty)$. Given any $ \varepsilon > 0$, there exists $N \in \N$ such that $ |a_n| < \frac{\varepsilon}{M}$ for any $n \geq N$.

    Thus, as for each $x \in [a,b]$ there is at most one $n_0 \geq N$ such that $x \in \supp(u_{n_0})$,
$$ \left| \sum_{n=N}^{\infty} a_n u_n(x)  \right|     =  |a_{n_0} u_{n_0}(x)|  \leq  M \cdot  \frac{\varepsilon}{M} = \varepsilon.$$
    Hence, the uniformly convergence of the series on $[a,b]$ is obtained.

    Conversely, as the series is uniformly convergent, we have $a_n u_n(x) \to 0$ ($n \to \infty$) uniformly on $[a,b]$. Because $(u_n)_n \in \mathcal{F}$, then $L:= \inf_n ||u_n||_\infty >0$. Given any $\varepsilon >0$, there is $n_0 \in \N$ such that $|a_n u_n(x)| < \varepsilon \cdot L $ for any $x \in [a,b]$ and any $n \geq n_0$. But for any $n \in \N$, there exists, by continuity, a point $x_n \in [a,b]$ such that $|u_n(x_n)| = ||u_n||_\infty \geq L$, hence for $n \geq n_0$ we get
    $$|a_n| = \frac{|a_n u_n(x_n)|}{|u_n(x_n)|} \leq  \frac{\varepsilon \cdot L}{L }= \varepsilon,$$
    and $a_n \to 0 $ ($n \to \infty$).

(c)  As $(u_n)_n \in\mathcal{F}$, we have
\begin{equation}
0< L:= \inf_n ||u_n||_\infty \leq M:= \sup_n ||u_n||_\infty <+ \infty.
\end{equation}
Thus, by the comparison test $(a_n)_n \in \ell_1$ if and only if $(||a_nu_n||_\infty)_n \in \ell_1$, and we have (c).
\end{proof}

As a consequence of the above lemma we obtain the next fact.
\begin{corollary} \label{AMW}
For any sequence of functions $(u_n)_n \in \mathcal{F}$ and any sequence of scalars $(a_n)_n \in c_0\setminus \ell_1$, the sequence of continuous functions $(a_nu_n)_n$ is in $\mathcal{AMW}$.
\end{corollary}
In particular, this corollary will allow us to provide a wide plethora of sequences exhibiting this strange behaviour.
In the following example we show that from any non null function $f \in C([a,b])$, many sequences in $\mathcal{F}$ can be constructed.
\begin{example}\label{sucesion}
Let $f\in C([a,b])\setminus \{0\}$. Let $\Lambda := (\alpha_n)_n$ be any sequence of scalars such that
$$a=\alpha_1 < \alpha_2 < \cdots < \alpha_{n-1} <\alpha_n < \cdots  \longrightarrow b \quad (n \to \infty).$$
For each $n \in \N$ we define the function $u^{\Lambda, f}_n : [a,b] \to \R $ by
$$u^{\Lambda, f}_n(x) := \left\{
\begin{array}{ll}
\displaystyle f(a) \cdot \frac{x-\alpha_{3n-2}}{\alpha_{3n-1}-\alpha_{3n-2}} & \text{if } \  x \in [\alpha_{3n-2}, \alpha_{3n-1}] \\[1em]
\displaystyle f\left(a + \frac{b-a}{\alpha_{3n}-\alpha_{3n-1}}(x-\alpha_{3n-1}) \right) & \text{if } \ x \in [\alpha_{3n-1}, \alpha_{3n}] \\[1em]
\displaystyle f(b)\cdot \frac{\alpha_{3n+1}-x}{\alpha_{3n+1}-\alpha_{3n}} & \text{if } \ x \in [\alpha_{3n}, \alpha_{3n+1}] \\[1em]
0 & \text{otherwise}.
\end{array}
\right.
$$
It is clear that $u^{\Lambda, f}_n \in C([a,b])$ and $\supp (u^{\Lambda, f}_n) \subset (\alpha_{3n-2} , \alpha_{3n+1} ) $ for each $n \in \N$, so the supports of the $u^{\Lambda, f}_n$'s are pairwise disjoint. In addition, $||u^{\Lambda, f}_n ||_\infty = ||f||_\infty \in (0, + \infty)$ for any $n\in \N$. So, trivially, $(u^{\Lambda, f}_n)_n \in \mathcal{F}$. Therefore, for any prefixed increasing sequence of scalars $\Lambda$ as above, we have defined an injective mapping
$$
\begin{array}{cccc}
\mathcal{J}_{\Lambda} : & C([a,b]) \setminus \{ 0 \} & \longrightarrow & \mathcal{F} \\
   & f & \longmapsto &\mathcal{J}_{\Lambda} (f):= (u_n^{\Lambda,f})_n,
   \end{array}
$$
satisfying the following properties:
\begin{enumerate}
\item $u_n^{\Lambda ,f } (\alpha_{3n-1}) =f(a)$ and $u_n^{\Lambda ,f } (\alpha_{3n}) =f(b)$ for any $n\in \N$ and any $f \in C([a,b]) \setminus \{0\}$.
\item $\supp(u_n^{\Lambda ,f }) \subset (\alpha_{3n-2}, \alpha_{3n+1})$ for any $f \in C([a,b]) \setminus \{0\}$ and any $n \in \N$.
\item For any $n \in \N$ there exists a linear affine transformation $\tau_n$ such that $\tau_n ([\alpha_{3n-1}, \alpha_{3n}])=[a,b]$ and $u_n^{\Lambda , f} = f \circ \tau_n$ on $[\alpha_{3n-1}, \alpha_{3n}]$ for each $f \in C([a,b]) \setminus \{ 0\}$.
\item $||u_n^{\Lambda ,f }||_\infty = ||f||_\infty $ for any $n \in \N$ and any $f \in C([a,b]) \setminus \{0 \} $.
\end{enumerate}
Observe that if we define $\mathcal{J}_{\Lambda} (0)$ as the null sequence, we have that $\mathcal{J}_{\Lambda} $ is a linear and injective mapping from $C([a,b])$ to $\mathcal{F} \cup \{0\}$. In particular if $\mathcal{C}$ is a linear vector subspace in $C([a,b])$ with dimension $\kappa$, then $\mathcal{J}_\Lambda (\mathcal{C})$ is a linear vector subspace in $\mathcal{F} \cup \{0 \}$ with dimension $\kappa$.
\end{example}
Furthermore, let $\ell_\infty (C([a,b]))$ be the vector space of all uniformly bounded sequences of continuous functions in $[a,b]$, which is a Banach space when endowed with the uniform supremum norm $$||(f_n)_n||_{\ell_\infty} := \sup_n ||f_n||_\infty.$$
Observe that $\mathcal{F} \subset \ell_\infty (C([a,b]))$, and for any sequence $\Lambda =(\alpha_n)_n$ as above and for each function $f\in C([a,b])$ we have $\mathcal{J}_\Lambda (f)  \in {\ell_\infty} (C([a,b]))$ and
$$|| \mathcal{J}_\Lambda (f) ||_{\ell_\infty} =\sup_n ||u_n^{\Lambda ,f } ||_\infty = ||f||_\infty.$$
Thus, we have obtained the following useful statement.
\begin{proposition} \label{isometry}
Let $\Lambda =(\alpha_n)_n \subset \R$ with $a=\alpha_1 < \alpha_2 < \cdots < \alpha_{n-1} <\alpha_n < \cdots  \longrightarrow b \quad (n \to \infty).$ Then the mapping
$$
\begin{array}{cccc}
\mathcal{J}_{\Lambda} : & C([a,b]) & \longrightarrow & \ell_\infty (C([a,b])) \\
   & f & \longmapsto &\mathcal{J}_{\Lambda} (f):= (u_n^{\Lambda,f})_n,
   \end{array}
$$
is a linear isometry whose rank is contained in $\mathcal{F} \cup \{0 \}$. In particular, the normed spaces $(C([a,b]), || \cdot ||_\infty)$ and $(\mathcal{J}_{\Lambda} (C([a,b])), || \cdot ||_{\ell^\infty})$ are isometrically isomorphic.
\end{proposition}

\section{Lineability and spaceability in $\mathcal{AMW}$}

In this section we show that $\mathcal{AMW}$ is maximal dense lineable but at the same time we provide many examples of concrete linear structures in $\mathcal{AMW}$.

\begin{theorem} \label{Thlin1}
Let M be a linear subspace such that $M \setminus \{0\} \subset c_0 \setminus \ell_1$ and $\dim (M)= \kappa$. Let $I$ be a set with $\card (I)= \kappa$ and $\{ (a_n^i)_n \}_{i\in I}$ be an algebraic basis of $M$. Then for any sequence of functions $(u_n)_n \in \mathcal{F}$, the set
$$\mathcal{M} := \{ (a_nu_n)_n : \ (a_n)_n \in M\}$$
is a linear subspace of dimension $\kappa$, the family $\{(a_n^i u_n )_n\}_{i\in I}$ is an algebraic basis of $\mathcal{M}$ and $\mathcal{M} \setminus \{0\} \subset \mathcal{AMW}$.
\end{theorem}
\begin{proof}
It is plain that the mapping $\Phi : (a_n)_n \in c_0 \longmapsto (a_nu_n)_n \in C([a,b])^{\N}$ is linear. Moreover, the fact that $u_n \ne 0$ for all $n \in \N$ implies that $\Phi$ is one-to-one. Then all conclusions follow from the equality $\mathcal{M} =\Phi (M)$ and from Corollary \ref{AMW}.
\end{proof}

By following a dual way (that is, fixing this time a sequence of $c_0 \setminus \ell_1$), we are able to obtain subspaces of $\mathcal{AMW}$.
\begin{theorem} \label{Thlin2}
Let $(a_n)_n \in c_0 \setminus \ell_1$ with $a_n \ne 0 $ for all $n \in \N$. Let $U$ be a linear subspace of dimension $\kappa$, having an algebraic basis $\{(u_n^i)_n\}_{i\in I}$ with $\card(I)= \kappa$, such that $U \setminus \{0\} \subset \mathcal{F}$. Then
$$\mathcal{U} := \{ (a_nu_n)_n: \ (u_n)_n \in U \}$$
is a linear subspace of dimension $\kappa$, the family $\{(a_nu_n^i)_n \}_{i\in I}$ is an algebraic basis of $\mathcal{U}$, and $\mathcal{U} \setminus \{0\} \subset \mathcal{AMW}$.
\end{theorem}
\begin{proof}
It is similar to the proof of Theorem \ref{Thlin1}, except that this time one considers the mapping $\Psi: (u_n)_n \in U \longmapsto (a_nu_n)_n \in C([a,b])^{\N} $. Then $\Psi$ is linear and injective (because $a_n \ne 0 $ for each $n \in \N $), $\mathcal{U} = \Psi (U)$, and Corollary \ref{AMW} applies again.
\end{proof}
Recall that in any separable, metrizable, complete topological vector space (as it is the case of $c_0$, $\ell_p$, $C([a,b])$ or $C([a,b])^\N $) the maximal dimension of any linear vector subspace is the dimension of the continuum $\mathfrak{c}$. In particular if we consider in Theorem \ref{Thlin1} any linear vector subspace $M$ in $c_0\setminus \ell_1$ of dimension $\mathfrak{c}$, for instance $M= \textrm{span} \left\{ \left( \frac{1}{n^c} \right)_n : \ c \in (0,1)\right\}$), we obtain:
\begin{corollary} \label{colin1}
The set $\mathcal{AMW}$ is maximal lineable in $C([a,b])^\N$.
\end{corollary}
We can obtain the same from Theorem \ref{Thlin2} just considering the linear subspace in $\mathcal{F} \cup \{0\} $ given by $U=\text{span} \{ \mathcal{J}_\Lambda ( x^c ):  c \in (0,+ \infty)\}$, where $\Lambda$ is any prefixed sequence of scalars as in Example \ref{sucesion}.

We denote by $c_0(C([a,b]))$ the topological vector space of all sequences $(f_n)_n \in C([a,b])^{\N}$ such that $||f_n||_\infty \to 0 $ ($n \to \infty$). It is known that $(c_0(C([a,b])), || \cdot ||_{\ell^\infty} )$ is a separable Banach space. It is clear that $\mathcal{AMW} \subset c_0(C([a,b]))$, so if we take into account the topological structure of $c_0(C([a,b]))$, we can focus our attention on density properties enjoyed by $\mathcal{AMW}$. A first positive result in this direction is given by the application of the following lemma due to Bernal \cite{Berstudia} (see also \cite[Chapter 4]{librolineabilidad}, \cite{topology} and \cite{BerOr}).

\begin{lemma}\label{denso}
Let $X$ be a separable metrizable topological vector space, $A \subset X$ maximal lineable and $B \subset X$ dense lineable in $X$ with $A \cap B = \varnothing$. If $A + B \subset A$ then $A$ is maximal dense lineable in $X$.
\end{lemma}

\begin{theorem}
The family $\mathcal{AMW}$ of Anti-M Weierstrass sequences is maximal dense lineable in $c_0( C([a,b]) )$.
\end{theorem}

\begin{proof}
From Corollary \ref{colin1}, the family $\mathcal{AMW}$ is maximal lineable. Now, the family
$$c_{00}(C([a,b])):= \{ (f_n)_n \in C([a,b])^{\N}: \ \text{there is } N \text{ such that } f_n=0 \text{ for any } n \geq N \}$$
is a dense linear subspace of $c_0(C([a,b]))$ and so, trivially, it is a dense lineable subset of $c_0(C([a,b]))$. For a fixed $(f_n)_n \in \mathcal{AMW}$, each sequence $(g_n)_n \in c_{00}(C([a,b]))$ only modifies, under addition, a finite number of components of $(f_n)_n$. So $(f_n + g_n )_n \in \mathcal{AMW}$. In other words, $c_{00}(C([a,b])) + \mathcal{AMW} \subset \mathcal{AMW}$. Now, an application of Lemma \ref{denso}, with $X:=c_0(C([a,b]))$, $A:= \mathcal{AMW}$ and $ B:= c_{00}(C([a,b]))$, gives us the maximal dense-lineability of $\mathcal{AMW}$ in $c_0(C([a,b]))$.
\end{proof}

Finally, we are also able to get the spaceability of the anti-M Weierstrass family of sequences. For this, the linear isometric property of the mapping $\mathcal{J}_\Lambda$ and the (trivial) spaceability of $C([a,b])$ (see \cite{librolineabilidad}, \cite{survey}) play an important role.

\begin{theorem} \label{spaceability}
The family $\mathcal{AMW}$ of anti-M Weierstrass sequences is spaceable in $c_0( \mathcal{C}([a,b]) )$. In fact, for any $\Lambda =(\alpha_n)_n$ with $a=\alpha_1 < \alpha_2 < \cdots < \alpha_n < \cdots \to b \ (n \to \infty)$, any sequence $(a_n)_n \in c_0 \setminus \ell_1$ and any infinite dimensional closed vector subspace $\mathcal{C}$ of $C([a,b])$, the set
$$\mathcal{M} := \{ (a_n u_n^{\Lambda ,f })_n: \ f \in \mathcal{C} \}$$
is an infinite dimensional closed vector subspace of $(c_0 (C([a,b])), || \cdot ||_{\ell_\infty})$ such that $\mathcal{M}\setminus\{0\} \subset \mathcal{AMW} $.
\end{theorem}

\begin{proof}
Let $\Lambda$, $(a_n)_n$ and $\mathcal{C}$ as above. It is plain that $\mathcal{M}$ is a linear space. From Proposition \ref{isometry}, $\mathcal{J}_\Lambda$ is a linear isometry, so $\mathcal{J}_{\Lambda} (\mathcal{C})$ is an infinite dimensional closed vector space of $\ell_\infty (C([a,b]))$ such that $  \mathcal{J}_{\Lambda} (\mathcal{C}) \setminus \{0 \} \subset \mathcal{F}$. But $(a_n)_n \in c_0$, so there exists $L>0$ such that $|a_n| \leq L$ for all $n \in \N$, and for any $f \in \mathcal{C}$ we have
$$||(a_n u_n^{\Lambda , f } )_n ||_{\ell_\infty} = \sup_{n \in \N} ||a_n u_n^{\Lambda ,f} ||_\infty  \leq L \cdot || (u_n^{\Lambda ,f})_n ||_{\ell_\infty},$$
and
$$||a_n u_n^{\Lambda , f } ||_\infty = |a_n | \cdot ||f||_\infty \to 0 \qquad (n \to \infty).$$
Therefore $\mathcal{M}$ is a closed vector subspace of $(\ell_\infty (C([a,b])), ||\cdot ||_{\ell_\infty})$ that is contained in $c_0 (C([a,b]))$. But $c_0 (C([a,b]))$ is closed in $\ell_\infty (C([a,b]))$, hence $\mathcal{M}$ is a closed vector subspace of $(c_0 (C([a,b])), || \cdot ||_{\ell_\infty})$. Of course, $\mathcal{M}$ is infinite dimensional because $\mathcal{C}$ is, there are infinitelly many $a_n \ne 0$ (recall $(a_n)_n \notin \ell_1$) and $\mathcal{J}_\Lambda $ is a linear isometry such that $u_n^{\Lambda , f} =0$ for some $n \in \N $ if and only if $f =0$ in $[a,b]$.

Finally, as $(a_n)_n \in c_0 \setminus \ell_1$ and $\mathcal{J}_\Lambda (\mathcal{C}) \setminus \{ 0 \}  \subset \mathcal{F}$, we get from Corollary \ref{AMW}, $\mathcal{M}\setminus \{0 \} \subset \mathcal{AMW}$.
\end{proof}
\begin{remark}
Observe that, for any sequence $\Lambda $ as above and any $(a_n)_n \in c_0 \setminus \ell_1$ we have that for $L:= \max_{n \in \N} |a_n| \in (0, + \infty )$,
$$\begin{array}{rcl}
||\lambda_1 (a_n u_n^{\Lambda , f_{1}})_n + \dots + \lambda_s (a_n u_n^{\Lambda , f_{s}})_n ||_{\ell_\infty} & = &|| (a_n u_n^{\Lambda ,(\lambda_1  f_{1}+ \cdots+ \lambda_{s} f_{s})})_n ||_{\ell_\infty} \\[1em]
  & =  & L \cdot ||\lambda_1  f_{1}+ \cdots +\lambda_{s} f_{s}||_\infty,
\end{array}
$$
where $\lambda_1, \dots ,\lambda_s \in \R$ and $f_1, \dots , f_s \in C([a,b])$. In particular, we can drop the restriction $a_n \ne 0$ in Theorem \ref{Thlin2} and we can state the algebraic basis of the subspaces in $\mathcal{AMW}$. Specifically:

(1) If $\{ f_i \}_{i\in I} \subset C([a,b])$ are linearly independent, then the family $M= \{ (a_n u_n^{\Lambda , f_i })_n: \ i \in I \}$ is linearly independent and $\text{span} (M) \setminus \{0\} \subset \mathcal{AMW} $;

(2) If $\overline{\text{span}}\{ f_i : \ i\in I\}$ is a closed vector subspace of $C([a,b])$ with dimension $\kappa$, then the closed vector subspace $\mathcal{M}:= \overline{\text{span}}\{ (a_n u_n^{\Lambda , f_i})_n : \ i\in I\}$ of $c_0(C([a,b]))$ has the same dimension $\kappa$ , and $\mathcal{M} \setminus \{0\} \subset \mathcal{AMW} $.
\end{remark}

\section{Algebrability of $\mathcal{AMW}$}

In this section we provide several ways to generate free algebras in $\mathcal{AMW}$. But before this we state two technical lemmas.

\begin{lemma} \label{lemmaalg1}
Let $\mathcal{U}$ be a free algebra in $C([a,b])$, generated by $U :=\{ u_i \}_{i \in I}$. Then, for any family $P=\{p_i\}_{i \in I}$ of polynomials of degree exactly $1$, the set $U_P:= \{ p_i \circ u_i\}_{i \in I}$ is a generator system of a free algebra in $C([a,b])$.
\end{lemma}

\begin{proof}
  Let $\mathcal{U_P}$ be the algebra generated by $U_P$. Let us see that $\mathcal{U_P}$ is free. By hypothesis, $p_i(x)= \alpha_i x + \beta_i$ with $\alpha_i, \beta_i \in \R$, $\alpha_i \ne 0$, for each $i \in I$. For every $F \in \mathcal{U_P}$, there is a polynomial $P(x_1, \dots ,x_N) = \sum_{{\bf j} \in J} \lambda_{{\bf j}} x_1^{i_1} \cdots x_N^{i_N}$ in $N$ ($N\in\N$) real variables, where $J \subset \N_0^N \setminus \{(0,\dots ,0)\}$ is finite, $\lambda_{\bf j} \in \R \setminus \{0 \}$, ${\bf j}=(j_1, \dots ,j_N) \in  J$, and such that $F=P(p_{i_1} \circ u_{i_1}, \dots, p_{i_N} \circ u_{i_N})$ for some pairwise different $i_1, \dots, i_N \in I$. Now, we relabel $p_{i_l} \circ u_{i_l} =: p_l \circ u_l$ for $l=1, \dots, N$. Then, for every $x \in [a,b]$, we have
  $$F(x)= \sum_{{\bf j} \in J} \lambda_{\bf j} (p_1(u_1(x)))^{j_1} \cdots (p_N(u_N(x)))^{j_N}=$$
$$  \sum_{{\bf j} \in J} \lambda_{\bf j}  (\alpha_1 u_1(x)+ \beta_1)^{j_1} \cdots (\alpha_N u_N(x)+ \beta_N)^{j_N}=$$
$$ \sum_{{\bf j} \in J} \lambda_{{\bf j}} \sum_{l_1=0}^{j_1} \binom{j_1}{l_1} \beta_1^{l_1} \cdot (\alpha_1u_1(x))^{j_1-l_1} \cdots \sum_{l_N=0}^{j_N} \binom{j_N}{l_N}  \beta_N^{l_N} \cdot (\alpha_N u_N(x))^{j_N-l_N}=$$
$$\sum_{{\bf j} \in J}  \sum_{l_1=0}^{j_1} \cdots \sum_{l_N=0}^{j_N} \lambda_{\bf j} \left( \binom{j_1}{l_1} (\beta_1^{l_1}  \cdot (\alpha_1 u_1(x))^{j_1-l_1}) \cdots \binom{j_N}{l_N} (\beta_N^{l_N} \cdot (\alpha_N u_N(x))^{j_N-l_N}) \right)=$$
$$\sum_{{\bf j} \in J}  \sum_{l_1=0}^{j_1} \cdots \sum_{l_N=0}^{j_N} \left( \lambda_{\bf j} \prod_{k=1}^{N}\binom{j_k}{l_k} \beta_k^{l_k}\cdot \alpha_k^{j_k -l_k} \right)   u_1(x)^{j_1-l_1} \cdots  u_N(x)^{j_N-l_N}.$$

But $\mathcal{U}$ is free, so if $F \equiv 0$ then $\lambda_{\bf j} \prod_{k=1}^{N}\binom{j_k}{l_k} \beta_k^{l_k}\cdot \alpha_k^{j_k -l_k} =0$
for any ${\bf j}=(j_1, \dots, j_N) \in J$, any $l_k  \in \{0, \dots, j_k\}$ and any $k \in \{1, \dots ,N\}$. In particular, by taking $l_k=0$ for any $k=1, \dots, N$ we obtain $\lambda_{\bf j} \cdot \left( \alpha_1^{j_1} \cdots \alpha_N^{j_N}\right)=0$ for any ${\bf j} \in J$, and since $\alpha_i \ne 0$ for all $i \in I$ we get $\lambda_{{\bf j}} =0$ for all ${\bf j} \in J$. Consequently $P \equiv 0$, as required.
\end{proof}

\begin{lemma} \label{lemmaalg2}
Let $(a_n)_n \in c_0 \setminus \bigcup_{p \geq 1} \ell_p$. Assume that $P$ is a polynomial with real coefficients and without constant term. Then
$$(P(a_n))_n \in c_0 \setminus \ell_1.$$
\end{lemma}

\begin{proof}
We can write $P(x)= \sum_{j=0}^m p_j x^{t+j}$, where $t \in \N$, $p_j \in \R$, $p_0 \ne 0$. Then, because $(a_n)_n \in c_0$, we have $(P(a_n))_n \in c_0$. Moreover, $(a_n^t)_n \not\in \ell_1$, so there are infinitelly many $n$ such that $a_n^t \ne 0$ (without loss of generality we may assume $a_n \ne 0$ for all $n$), and
$$\lim_{n \to \infty } \frac{|P(a_n)|}{|a_n^t|}= \lim_{n \to \infty} \left| \sum_{j=0}^m p_j a_n^j \right| = |p_0| >0.$$
Thus, the result follows from the comparison test.
\end{proof}

We endow the vector space $C([a,b])^{\N}$ with the structure of linear algebra given by the coordenatewise multiplication: $(f_n)_n \cdot (g_n)_n= (f_n\cdot  g_n)_n$. Similarly, coordenatewise multiplication will be considered in the algebraic structure of the sequence space $c_0$.
\begin{theorem} \label{thalg1}
Assume that $(a_n)_n \in c_0 \setminus \bigcup_{p \geq 1} \ell_p$. Let $G :=\{ g_i \}_{i \in I} $ be a $($minimal$\,)$ generator system of a free algebra in $C([a,b])$. Let $\Lambda =(\alpha_n)_n \subset [a,b]$ such that $a=\alpha_1 < \alpha_2 < \cdots < \alpha_n < \alpha_{n+1} < \cdots \to b$  $(n \to \infty)$. Consider the family
$$U:=\{ (a_n u^i_n)_n: (u^i_n)_n= \mathcal{J}_{\Lambda} (\gamma_i g_i + 1)\}    ,$$
where $\gamma_i := \frac{1}{g_i(a)}$ if $g_i(a) \ne 0$ or $1$ if $g_i(a) = 0$.
Then $U$ is the $($minimal$\,)$ generator system of a free algebra in $\mathcal{AMW} \cup \{0\}$.
\end{theorem}
\begin{proof}
Let $\mathcal{A}$ be the algebra generated by $U$, that is, $(F_n)_n \in \mathcal{A}$ if there exist a polynomial $P(x_1, \dots , x_N)$ in $N$ $(N \in \N)$ real variables and some pairwise different $i_1, \dots i_N \in I$ such that
$$F_n= P(a_nu_n^{i_1}, a_nu_n^{i_2}, \dots , a_nu_n^{i_N}) \qquad (n \in \N).$$
Hence, by relabeling $(u_n^{i_l})_n=(u_n^l)$ ($l=1, \dots ,N$), there exist a nonempty finite set $J \subset \N_0^N \setminus \{(0,\dots ,0)\}$ and scalars $\lambda_{\bf j} \in \R \setminus \{0\}$ for ${\bf j}=(j_1, \dots, j_N) \in J$ such that for each $n \in \N$ and each $x \in [a,b]$,
\begin{equation} \label{0}
F_n(x)= \sum_{{\bf j} \in J} \lambda_{{\bf j}} (a_nu_n^1)^{j_1}(x) \cdots (a_n u_n^N)^{j_N}(x)=
\end{equation}
$$\sum_{{\bf j} \in J} \lambda_{{\bf j}} a_n^{|{\bf j}|} u_n^1(x)^{j_1} \cdots u_n^N(x)^{j_N},$$
where we have denoted $|{\bf j}| := j_1 + \cdots + j_N$.

But, by the definition of $\mathcal{J}_\Lambda$ (see property (3) in Example \ref{sucesion}), for each $n \in \N $ there is a linear affine transformation $\tau_n $ such that $\tau_n([\alpha_{3n-1}, \alpha_{3n}])= [a,b]$ and $u_n^l(x)=\mathcal{J}_{\Lambda} (\gamma_l g_l +1)(x) = \left( (\gamma_l g_l +1)(\tau_n(x))\right)_{n \geq 1}$ for any $x \in [\alpha_{3n-1}, \alpha_{3n}]$ and each $l=1, \dots, N$.
So, for every $n \in \N$,
$$F_n(x)= \sum_{{\bf j} \in J} \lambda_{{\bf j}} a_n^{|{\bf j}|} ((\gamma_1 g_1 +1)(\tau_n(x)))^{j_1} \cdots (\gamma_N g_N +1)(\tau_n (x)))^{j_N} \qquad{\rm for } \ x \in [\alpha_{3n-1}, \alpha_{3n}].$$
Therefore, if $F_n=0 $ in $[a,b]$, we get
\begin{equation} \label{algebra}
\sum_{{\bf j} \in J} \lambda_{{\bf j}} a_n^{|{\bf j}|} ((\gamma_1 g_1 +1)(w))^{j_1} \cdots (\gamma_N g_N +1)(w))^{j_N}=0 \qquad {\rm for \ all } \ w \in [a, b].
\end{equation}
Now, as $\{ g_i\}_{i \in I}$ generates a free algebra in $C([a,b])$, from Lemma \ref{lemmaalg1}, the equatity \eqref{algebra} and the fact that there are infinitely many $a_n \ne 0$, we get $\lambda_{\bf j}= 0$ for ${\bf j} \in J$ and the algebra $\mathcal{A}$ is free. Observe that, trivially, we also obtain that the dimension of $\mathcal{A}$ is the dimension of the algebra generated by $\{g_i\}_{i \in I}$ in $C([a,b])$.

It only remains to show that $\mathcal{A} \setminus\{ 0 \} \subset \mathcal{AMW}$, that is, that any sequence $(F_n)_n$ as above is Anti M-Weierstrass. Note that, again by the definition of the application $\mathcal{J}_\Lambda$ (see Example \ref{sucesion}, properties (2) and (4)), we have:
\begin{equation} \label{11}
\supp (u_n^l) \subset [\alpha_{3n-2}, \alpha_{3n+1}] \text{ for each } l =1, \dots ,N \text{ and } n \in \N,
\end{equation}
as well as
\begin{equation}\label{12}
||u_n^l||_\infty = ||\gamma_l g_l + 1||_\infty \in (0, + \infty) \text{ for each }  l =1, \dots ,N \text{ and each } n \in \N.
\end{equation}
Therefore, thanks to \eqref{11},
\begin{equation} \label{1}
\supp(u_n^1 \dots u_n^N) \cap \supp(u_m^1 \dots u_m^N) =\varnothing \qquad {\rm for} \ n\ne m,
\end{equation}
and we have the absolute convergence of $\sum_{n=1}^{\infty} F_n$. Now, by \eqref{12},
\begin{equation} \label{2}
\sup_{n \in \N} ||u_n^1 \cdots u_n^N||_{\infty} \leq ||\gamma_1 g_1+1||_\infty  \cdots ||\gamma_N g_N+1||_\infty =:C \in (0, \infty) ,
\end{equation}
and, as $(a_n)_n \in c_0$ by \eqref{0}, \eqref{11}, and \eqref{2}, the series $\sum_{n=1}^{\infty} F_n(x)$ converges uniformly in $[a,b]$ (note again that by \eqref{11} the supports of the $u_n^l$, $n \in \N$, are pairwise disjoint for each $l$, and this fact is crucial).
Finally, again by the definition of $\mathcal{J}_\Lambda$ (see Example \ref{sucesion}, property (1)),
\begin{equation} \label{3}
u_n^l (\alpha_{3n-1}) = \gamma_l g_l(a) +1 = : \delta = 1 \text{ or } 2 \text{ for any } l=1, \dots ,N \text{ and any } n \in \N.
\end{equation}
Now, by \eqref{0} and \eqref{3}, we obtain for each $n \in \N$ that
$$|F_n(\alpha_{3n-1})| = \left|  \sum_{{\bf j} \in J} \lambda_{{\bf j}} (a_nu_n^1)^{j_1}(\alpha_{3n-1}) \cdots (a_n u_n^N)^{j_N}(\alpha_{3n-1})\right|=\left| \sum_{{\bf j} \in J} \lambda_{\bf j} \delta^{|{\bf j}|} a_n^{|{\bf j}|} \right|.$$
But $(a_n )_n\in c_0 \setminus \bigcup_{p \geq 1}\ell_p$, so by applying Lemma \ref{lemmaalg2} to the polynomial $P(x) = \sum_{{\bf j} \in J} \lambda_{\bf j} \delta^{|{\bf j}|} x^{|{\bf j}|}$, we get
$$\sum_{n=1}^\infty ||F_n||_\infty \geq \sum_{n=1}^\infty \left| F_n(\alpha_{3n-1}) \right|= \sum_{n=1}^{\infty} | P(a_n)| = + \infty .$$
Thus $(F_n)_n \in \mathcal{AMW}$, and we are done.
\end{proof}
Recall that the family $\mathcal{F}$ is defining at the beginning of Section 2.
\begin{theorem} \label{thalg2}
Let $\mathcal{L}$ be a free algebra in $c_0 \setminus \ell_1$ generated by $\{ (a_n^i)_n\}_{i\in I}$. Let $(u_n)_n \in \mathcal{F}$ such that $u_n \geq 0$ on $[a,b]$ for each $n \in \N$. Then the algebra $\mathcal{A}$ generated  by $\{ (a_n^i u_n)_n \}_{i \in I}$
is a free algebra in $\mathcal{AMW} \cup \{0\}$ with the same dimension as $\mathcal{L}$.
\end{theorem}
\begin{proof}
  Let $(F_n)_n \in \mathcal{A}$. Then
  $$F_n(x):= \sum_{{\bf j} \in J} \lambda_{\bf j} (a_n^1 u_n(x))^{j_1} \cdots  (a_n^N u_n(x))^{j_N}=$$
  $$  \sum_{{\bf j} \in J} \lambda_{\bf j} (a_n^1)^{j_1} \cdots  (a_n^N)^{j_N}\cdot  (u_n(x))^{|{\bf j}|} \qquad (x \in [a,b]),$$
  where $N \in \N$, $J \subset \N_0^N \setminus \{(0,\dots ,0)\}$ is nonempty and finite and $\lambda_{\bf j} \in \R \setminus \{0\}$ for any ${\bf j}=(j_1, \dots, j_N) \in J$ (again we have relabeled $(a_n^{i_1})_n, \dots, (a_n^{i_N})_n$ as $(a_n^{1})_n, \dots, (a_n^{N})_n$).

  As $(u_n)_n \in \mathcal{F}$, we have that $L:=\inf_n ||u_n ||_\infty >0$ and any $u_n$ ($n \in \N $) takes the value $0$ in $[a,b]$. Then, by continuity of $u_n$'s and the intermediate value property, for any $n \in \N$ there is $x_n \in \supp(u_n)$ such that $L= |u_n(x_n)|= u_n(x_n)$ (recall that $u_n \geq 0$ on $[a,b]$). Hence
  \begin{equation}\label{teo2eq1}
    F_n(x_n)=  \sum_{{\bf j} \in J} \lambda_{\bf j} (a_n^1)^{j_1} \cdots  (a_n^N)^{j_N}\cdot  L^{|{\bf j}|}  \qquad (n \in \N).
  \end{equation}
If $(F_n)_n \equiv 0$, by (\ref{teo2eq1}), for each $n \in \N$
$$ \sum_{{\bf j} \in J} \lambda_{\bf j} L^{|{\bf j}|} \cdot (a_n^1)^{j_1} \cdots  (a_n^N)^{j_N}   =0.$$
So, because of $L>0$ and the free condition of the algebra $\mathcal{L}$, we get $\lambda_{\bf j} =0$ for each ${\bf j} \in  J$, and $\mathcal{A}$ is free.
As $\mathcal{L} \subset  c_0$, then $((a_n^1)^{j_1} \cdots (a_n^N)^{j_N})_n \in c_0 $ for any ${\bf j} =(j_1, \dots, j_N ) \in J$. Trivially, $(u_n^p)_n \in \mathcal{F}$ for any $p \in \N$. Hence, $(F_n)_n$ is a finite linear combination of products of sequences in $c_0$ and sequences $(u_n^{|{\bf j}|})_n \in \mathcal{F}$. Now, by Lemma \ref{FA}, we have the absolute and uniform convergence of the series $\sum_{n=1}^{\infty} F_n(x)$.
It only rest to show that $ (||F_n||_\infty )_n \not\in \ell_1$ to obtain $(F_n) \in \mathcal{AMW}$ and finish the proof. By (\ref{teo2eq1}),
$$\sum_{n=1}^{\infty } ||F_n||_\infty \geq \sum_{n=1}^{\infty} |F_n(x_n)| = \sum_{n=1}^{\infty} \left| \sum_{{\bf j} \in J} \lambda_{\bf j} L^{|{\bf j}|} \cdot (a_n^1)^{j_1} \cdots  (a_n^N)^{j_N} \right| = + \infty,$$
because of $\mathcal{L}$ is an algebra in $c_0 \setminus \ell_1$.
 \end{proof}

Finally, observe that $(a_n)_n\in \ell_p$ if and only if $(|a_n|^p)_n\in \ell_1$. In particular, $\mathcal{L}$ is an algebra in $c_0 \setminus \ell_1$ if and only if $\mathcal{L}$ is an algebra in $c_0 \setminus \bigcup_{p\geq 1} \ell_p$. In \cite[Theorem 2]{BarGlab13}, it is showed the existence of a free algebra $\mathcal{L}$ in $c_0 \setminus \bigcup_{p \geq 1} l_p$ such that the cardinality of any system of generators is the continuum. In fact, the authors of \cite{BarGlab13} consider the algebra generated by $ \displaystyle \left\{\left(\frac{1}{\ln^c n} \right)_{n\geq 2} : \ c \in C \right\}$, where $C\subset (0, +\infty)$ is linearly independent set over the field $\Q$ and
$\card(C) = \mathfrak{c}$. If we apply Theorem \ref{thalg2} with this algebra $\mathcal{L}$, we get the next corollary.
\begin{corollary}
The family $\mathcal{AMW}$ is strongly $\mathfrak{c}$-algebrable.
\end{corollary}
The same above result also follows by applying Theorem \ref{thalg1} to the free algebra generated by $\displaystyle \left\{ \left(\frac{x-a}{b-a}\right)^c: c \in C \right\}$, where $C$ is as above. The free algebra generated by $\{ e^{cx} : c \in C\}$ also works.




\end{document}